\documentclass[reqno, 12pt]{amsart}

\usepackage{amsthm,amssymb,amstext,amscd,amsfonts,amsbsy,amsxtra,latexsym, 
amsmath, color,graphicx}
\usepackage{fullpage}
\usepackage[latin1]{inputenc}%

\definecolor{blue}{rgb}{0,0,1}
\definecolor{red}{rgb}{1,0,0}
\definecolor{green}{rgb}{0,1,0}

\title{On an Additive prime divisor Function\\ of Alladi and Erd\H os}
\author{Dorian Goldfeld}
\address{Department of Mathematics\\
  Columbia University\\
  New York, NY 10027}
\email{goldfeld@columbia.edu}

\newtheorem{theorem}{Theorem}[section]
\newtheorem{corollary}{Corollary}[theorem]
\newtheorem{lemma}[theorem]{Lemma}

\begin{document}

\maketitle
\dedicatory\centerline{\it This paper is dedicated to Krishna Alladi on the occasion of his $60^{th}$ birthday.}
\let\thefootnote\relax\footnote{This research was partially supported by NSA grant
     H98230-15-1-0035.}
\begin{abstract}
This paper discusses the additive prime divisor function $A(n) := \sum\limits_{p^\alpha || n} \alpha \, p$ which was introduced by Alladi and Erd\H os in 1977. It is shown that $A(n)$ is uniformly distributed (mod $q$) for any fixed integer $q > 1$ with an explicit bound for the error.
\end{abstract}

 \section{Introduction}
 
Let $n = \prod\limits_{i=1}^r p_i^{a_i}$  be the unique prime decomposition of a positive integer $n$. In 1977, Alladi and Erd\H os \cite{AE} introduced the additive function $$A(n) := \sum_{i=1}^r a_i \cdot p_i.$$ Among several other things they proved that $A(n)$ is uniformly distributed modulo 2. This was obtained from the identity
\begin{equation} \label{mod2dist}
\sum_{n=1}^\infty \frac{(-1)^{A(n)}}{n^s} \; = \; \frac{2^s+1}{2^s-1} \cdot \frac{\zeta(2s)}{\zeta(s)}
\end{equation}
together with the known zero free region for the Riemann zeta function.
As a consequence they proved that there exists a constant $c > 0$ such that
$$
\sum_{n\le x} (-1)^{A(n)} = \mathcal O\left( x \, e^{-c \sqrt{\log x \log\log x}}  \right),
$$
for $x \to \infty.$
\vskip 4pt
  In 1969 Delange \cite{Del} gave a necessary and sufficient condition for uniform distribution in progressions for integral valued additive functions  which easily implies that  $A(n)$ is uniformly distributed (mod $q$) for all $q \ge 2$ (although without a bound  for the error in the asymptotic formula).
  The main goal of this paper is to show  that $A(n)$ is uniformly distributed modulo $q$ for any integer $q \ge 2$ with an explicit bound for the error.
 \vskip 4pt  
  Unfortunately, it is not possible to obtain such a simple identity as in (\ref{mod2dist}) for the Dirichlet series
 $$\sum_{n=1}^\infty\; \frac{e^{2\pi i \frac{h A(n)}{q}}}{ n^{s}}$$
 when $q > 2$ and $h,q$ are coprime. Instead we require a representation involving a product of rational powers of Dirichlet L-functions which will have branch points at the zeros of the L-functions.
 
 The uniform distribution of $A(n)$  is a consequence of the following theorem (\ref{MainTheorem})  which is proved in \S 3.
 To state the theorem we require some standard notation. Let $\mu$ denote the Mobius function and let $\phi$ denote Euler's function. For any Dirichlet character $\chi\hskip -3pt\pmod{q}$ (with $q > 1$) let $\tau(\chi) = \sum\limits_{\ell\hskip-5pt\pmod{q}} \chi(\ell) e^{\frac{2\pi i\ell}{q}}$
  denote the associated Gauss sum and let $L(s, \chi)$ denote the Dirichlet L-function associated to $\chi.$ 
 \vskip 10pt
 
 \begin{theorem} \label{MainTheorem} 
 Let $h,q$ be fixed coprime integers with $q > 2.$ Then for $x \to \infty$ we have the asymptotic formula
 $$\sum_{n\le x} e^{2\pi i \frac{h A(n)}{q}} \; = \; \begin{cases}
 C_{h,q} \cdot x \,  (\log x)^{-1 + \frac{\mu(q)}{\phi(q)}} \Big(1 + \mathcal O\left( (\log x)^{-1}\right)\Big)  & \text{if} \; \mu(q) \ne 0,\\
 &\\
  \mathcal O\left( x\,  e^{-c_0\,\sqrt{\log x}}  \right)& \text{if} \; \mu(q) = 0,
 \end{cases}$$
 where $c_0>0$ is a constant depending at most on $h, q$,
  $$C_{h,q} =  \frac{V_{h,q}\cdot \sin\left(\frac{\mu(q) \,\pi }{\phi(q)}   \right)}{\pi}\, \Gamma\left(1 - \frac{\mu(q)}{\phi(q)}  \right) \underset{\chi\ne\chi_0}{\prod_{\chi\hskip-5pt\pmod{q}}} \;L(1, \,\chi)^{\frac{\tau(\overline{\chi})\chi(h)}{\phi(q)}},$$
 and 
 $$V_{h,q} :=  \exp\left[ - \frac{\mu(q)}{\phi(q)}\sum_{p\mid q}\sum_{k=1}^\infty \frac{1}{kp^{k}} \; + \; \sum_{p\,\mid \,q} \sum_{k=1}^\infty \frac{e^{\frac{2\pi i h p^k}{q}}}{k \, p^{k}}  \; + \; \sum_p \sum_{k=2}^\infty \frac{e^{\frac{2\pi i p h k}{q}} - e^{\frac{2\pi i p^k h}{q}}  }{k \, p^{k}}\right].$$

  \end{theorem}

Theorem \ref{MainTheorem} has the following easily proved corollary.

\begin{corollary} \label{corollary}
Let $q >1$ and let $h$ be an arbitrary integer. Then
$$\sum_{n\le x} e^{2\pi i \frac{h A(n)}{q}} = \mathcal O\left( \frac{x}{\sqrt{\log x}} \right).$$
\end{corollary}

 The above corollary can then be used to obtain the desired uniform distribution theorem.
 
 \begin{theorem} \label{thm2}
 Let $h,q$ be fixed integers with $q > 2.$ Then for $x\to\infty$, we have
 $$\underset {A(n)\; \equiv\; h\hskip-5pt\pmod{q}} {\sum_{n \,\le \, x}} \hskip -10pt 1 \; = \;\frac{x}{q} \; + \; \mathcal O\left( \frac{x}{\sqrt{\log x}} \right).$$
 \end{theorem}
 
We remark that the error term  in theorem \ref{thm2} can be replaced by a second order asymptotic term which is not uniformly distributed (mod $q$).

\vskip 5pt
The proof of theorem  (\ref{MainTheorem}) relies on explicitly constructing an L-function with coefficients of the form $e^{2\pi i \frac{h A(n)}{q}}$. It will turn out that this L-function will be a product of Dirichlet L-functions raised to complex powers. The techniques for obtaining asymptotic formulae and dealing with branch singularities arising from complex powers of ordinary L-series were first introduced by Selberg \cite{Sel}, and see also  Tenenbaum \cite{Ten}  for a very nice exposition with different applications.  
 In \cite{Ivic}, \cite{IE1}, \cite{IE2} one finds a larger class of  additive functions where these methods can also be applied yielding similar results but with different constants. \vskip 20pt

 \section{On the function $L(s, \psi_{h/q})$}

 Let $h,q$ be coprime integers integers with $q > 1$.  In this paper we shall investigate the completely multiplicative function $$\psi_{h/q}(n) := e^{\frac{2\pi i hA(n)}{q}}.$$
  
     Then the L-function associated to $\psi_{h/q}$ is defined by the absolutely convergent series
     
     \begin{equation} \label{Lfunction}
L(s, \psi_{h/q}) :=  \; \sum_{n=1}^\infty  \psi_{h/q}(n) n^{-s},
\end{equation}
in the region $\Re(s) > 1,$ and has an
 Euler product representation (product over rational primes) of the form
\begin{equation} \label{EulerProduct}
L(s, \psi_{h/q}) := \prod_p \left(1 - \frac{ e^{\frac{2\pi i h p}{q}} }{p^s}   \right)^{-1} .
\end{equation}

\vskip 8pt
 The Euler product  (\ref{EulerProduct}) converges absolutely to a non-vanishing function for $\Re(s) > 1.$ We would like to show it has analytic continuation to a larger region.

\vskip 10pt 
  \begin{lemma} \label{lemma1}  Let $\Re(s) > 1.$ Then
  $$\log\big(L(s, \psi_{h/q})\big) = \sum_p \sum_{k=1}^\infty \frac{e^{\frac{2\pi i h p^k}{q}}}{k \, p^{sk}} \; + \; T_{h,q}(s)$$
  where, for any $\epsilon > 0$, the function
 $$T_{h,q}(s) := \sum_p \sum_{k=2}^\infty \frac{e^{\frac{2\pi i p h k}{q}} - e^{\frac{2\pi i p^k h}{q}}  }{k \, p^{sk}}$$ is holomorphic for $\Re(s) > \frac12 + \epsilon$ and satisfies $|T_{h,q}(s)| = \mathcal O_\epsilon\left(1\right)$ where the $\mathcal O_\epsilon$-constant is  independent of $q$ and depends at most on $\epsilon$.

\end{lemma}

\begin{proof}
  Taking log's, we obtain
 \begin{align*} \label{logL} 
 \log\big( L(s, \psi_{h/q})  \big) & = \sum_p \sum_{k=1}^\infty \frac{e^{\frac{2\pi i p h k}{q}}}{k \, p^{sk}}\\
 & =  \sum_p \sum_{k=1}^\infty \frac{e^{\frac{2\pi i h p^k}{q}}}{k \, p^{sk}} \; + \;\sum_p \sum_{k=2}^\infty \frac{e^{\frac{2\pi i p h k}{q}} - e^{\frac{2\pi i p^k h}{q}}  }{k \, p^{sk}}.  
 \end{align*}
 Hence, we may take
 $$T_{h,q}(s) = \sum_p \sum_{k=2}^\infty \frac{e^{\frac{2\pi i p h k}{q}} - e^{\frac{2\pi i p^k h}{q}}  }{k \, p^{sk}},$$ which is easily seen to converge absolutely for $\Re(s) > \frac12.$
 \end{proof}
 
 \pagebreak
 
\vskip 10pt  
  For $q > 2,$ let $\chi$ denote a Dirichlet character $\hskip-4pt \pmod{q}$ with associated Gauss sum
  $\tau(\chi).$ We also let $\chi_0$ be the trivial character $\hskip-5pt\pmod{q}.$
  \vskip 8pt
  We require the following lemma.
  
  \begin{lemma} \label{lemma2}
Let $h, q\in \mathbf Z$  with $q > 2$ and $(h,q) = 1.$ Then
$$ e^{\frac{2\pi i h}{q}} = \left(\frac{1}{\phi(q)} \underset{\chi\ne\chi_0} {\sum_{\chi\hskip-5pt\pmod{q}}}  \tau(\chi)\cdot \overline{\chi(h)} \right) \; + \; \frac{\mu(q)}{\phi(q)}$$
\end{lemma}

\begin{proof}
Since $(h,q) = 1,$ it follows that for $\chi\hskip-3pt\pmod{q}$ with $\chi \ne \chi_0,$
$$\tau(\chi) \, \overline{\chi(h)} = \sum_{\ell=1}^q \chi(\ell) e^{\frac{2\pi i \ell h}{q}}.$$
This implies that
\begin{align*} \underset{\chi\ne\chi_0} {\sum_{\chi\hskip-5pt\pmod{q}}}  \tau(\chi)\, \overline{\chi(h)}  & = (\phi(q)-1) \,e^{\frac{2\pi i h}{q}} \; + \; \underset{(\ell,q)=1}{\sum_{\ell=2}^q }\left( \underset{\chi\ne\chi_0} {\sum_{\chi\hskip-5pt\pmod{q}}} \,\chi(\ell)\right) e^{\frac{2\pi i \ell  h}{q}}\\
& =  (\phi(q)-1) \,e^{\frac{2\pi i h}{q}} \; - \; \underset{(\ell,q)=1}{\sum_{\ell=1}^q } e^{\frac{2\pi i \ell h}{q}} \; + \; e^{\frac{2\pi i h}{q}}.
\end{align*}
The proof is completed upon noting that the Ramanujan sum on the right side above can be evaluated as
$$\underset{(\ell,q)=1}{\sum_{\ell=1}^q } e^{\frac{2\pi i \ell h}{q}} = \sum_{d\mid (q,h)} \mu\left(\frac{q}{d}\right) \,d \; = \; \mu(q).$$

\end{proof}

\vskip 5pt  
 \begin{theorem}  \label{Identity}
Let $s\in\mathbf C$ with $\Re(s)>1.$ Then we have the representation
$$L(s, \psi_{h/q}) =  \left(\underset{\chi\ne\chi_0}{\prod_{\chi\hskip-5pt\pmod{q}}} \;L(s,\overline{\chi})^{\frac{\tau(\chi)\overline{\chi(h)}}{\phi(q)}}\right) \cdot \zeta(s)^{\frac{\mu(q)}{\phi(q)}}\cdot e^{U_{h,q}(s)},$$
where
$$U_{h,q}(s) :=   - \frac{\mu(q)}{\phi(q)}\sum_{p\mid q}\sum_{k=1}^\infty \frac{1}{kp^{sk}} \; + \; \sum_{p\,\mid \,q} \sum_{k=1}^\infty \frac{e^{\frac{2\pi i h p^k}{q}}}{k \, p^{sk}}  \; + \; \sum_p \sum_{k=2}^\infty \frac{e^{\frac{2\pi i p h k}{q}} - e^{\frac{2\pi i p^k h}{q}}  }{k \, p^{sk}}$$
 \end{theorem}
 
 \begin{proof}
 
  If we combine lemmas (\ref{lemma1}) and (\ref{lemma2})  it follows that for $\Re(s) > 1$,
 \begin{align*} 
 \log\big( L(s, \psi_{h/q})  \big) & =  \sum_p \sum_{k=1}^\infty \frac{e^{\frac{2\pi i h p^k}{q}}}{k \, p^{sk}} \; + \; T_{h,q}(s)\\
 & = \sum_{p\,\nmid \,q} \sum_{k=1}^\infty \frac{e^{\frac{2\pi i h p^k}{q}}}{k \, p^{sk}} \; +\;\sum_{p\,\mid \,q} \sum_{k=1}^\infty \frac{e^{\frac{2\pi i h p^k}{q}}}{k \, p^{sk}}  + \; T_{h,q}(s)   \\
 & = \; \sum_{p\,\nmid \,q} \sum_{k=1}^\infty \frac{ \left(\frac{1}{\phi(q)} \underset{\chi\ne\chi_0} {\sum\limits_{\chi\hskip-5pt\pmod{q}}}  \tau(\chi)\cdot \overline{\chi(h\,p^k)}  \; + \; \frac{\mu(q)}{\phi(q)}\right)}{k \, p^{sk}} \; +\;\sum_{p\,\mid \,q} \sum_{k=1}^\infty \frac{e^{\frac{2\pi i h p^k}{q}}}{k \, p^{sk}}  + \; T_{h,q}(s). 
 \end{align*}
 
 Hence
 \begin{align*}
 \log\big( L(s, \psi_{h/q})  \big)
 & = \;\frac{1}{\phi(q)} \underset{\chi\ne\chi_0} {\sum_{\chi\hskip-5pt\pmod{q}}}  \tau(\chi) \overline{\chi(h)}\, \log(L(s, \overline{\chi}) \; + \;  \frac{\mu(q)}{\phi(q)} \log\big( \zeta(s)  \big)\nonumber\\ 
 &\hskip 40pt - \frac{\mu(q)}{\phi(q)}\sum_{p\mid q}\sum_{k=1}^\infty \frac{1}{kp^{sk}} \; + \; \sum_{p\,\mid \,q} \sum_{k=1}^\infty \frac{e^{\frac{2\pi i h p^k}{q}}}{k \, p^{sk}}  \; + \; T_{h,q}(s).
 \end{align*}
  The theorem immediately follows after taking exponentials.
 \end{proof}

  The representation of $L(s, \psi_{h/q})$ given in theorem \ref{Identity} allows one to analytically continue the function $L(s, \psi_{h/q})$ to a larger region which lies to the left of the line $\Re(s) = 1 +\varepsilon$ ($\varepsilon > 0$).  This is a region which does not include the branch points of  $L(s, \psi_{h/q})$ at the zeros and poles of $L(s,\chi), \zeta(s)$.

  \vskip 8pt 
  Assume that $q > 1$ and $\chi\hskip -3pt \pmod{q}$.  It is well known (see \cite{Dav}) that the Dirichlet L-functions $L(\sigma+it,\chi)$) do not vanish in the region
 \begin{equation} \label{ZeroFreeRegionL}
 \sigma \ge \begin{cases} 
 1 - \frac{c_1}{\log q|t|} & \text{if} \; |t| \ge 1,\\
 1 - \frac{c_2}{\log q} & \text{if} \; |t| \le 1, 
  \end{cases} \qquad\quad(\text{for absolute constants} \; c_1,c_2 >0),
 \end{equation}
 unless $\chi$ is the exceptional real character which has a simple real zero (Siegel zero) near $s = 1.$
  
  Similarly, $\zeta(\sigma+it)$ does not vanish for
 \begin{equation} \label{ZeroFreeRegionZ}
\sigma \ge 1 - \frac{c_3}{\log(|t| + 2)}, \qquad \qquad (\text{for an absolute constant} \;c_3 > 0).
\end{equation}
  
 Assume $q > 1$ and that there is no exceptional real character (mod $q$).  It follows from (\ref{ZeroFreeRegionL}) and (\ref{ZeroFreeRegionZ}) that $L(s, \psi_{h/q})$ is holomorphic in the region to the right of the contour $\mathcal C_q$ displayed in  Figure 1.

    \centerline{\hskip-101pt\includegraphics[width=65mm]{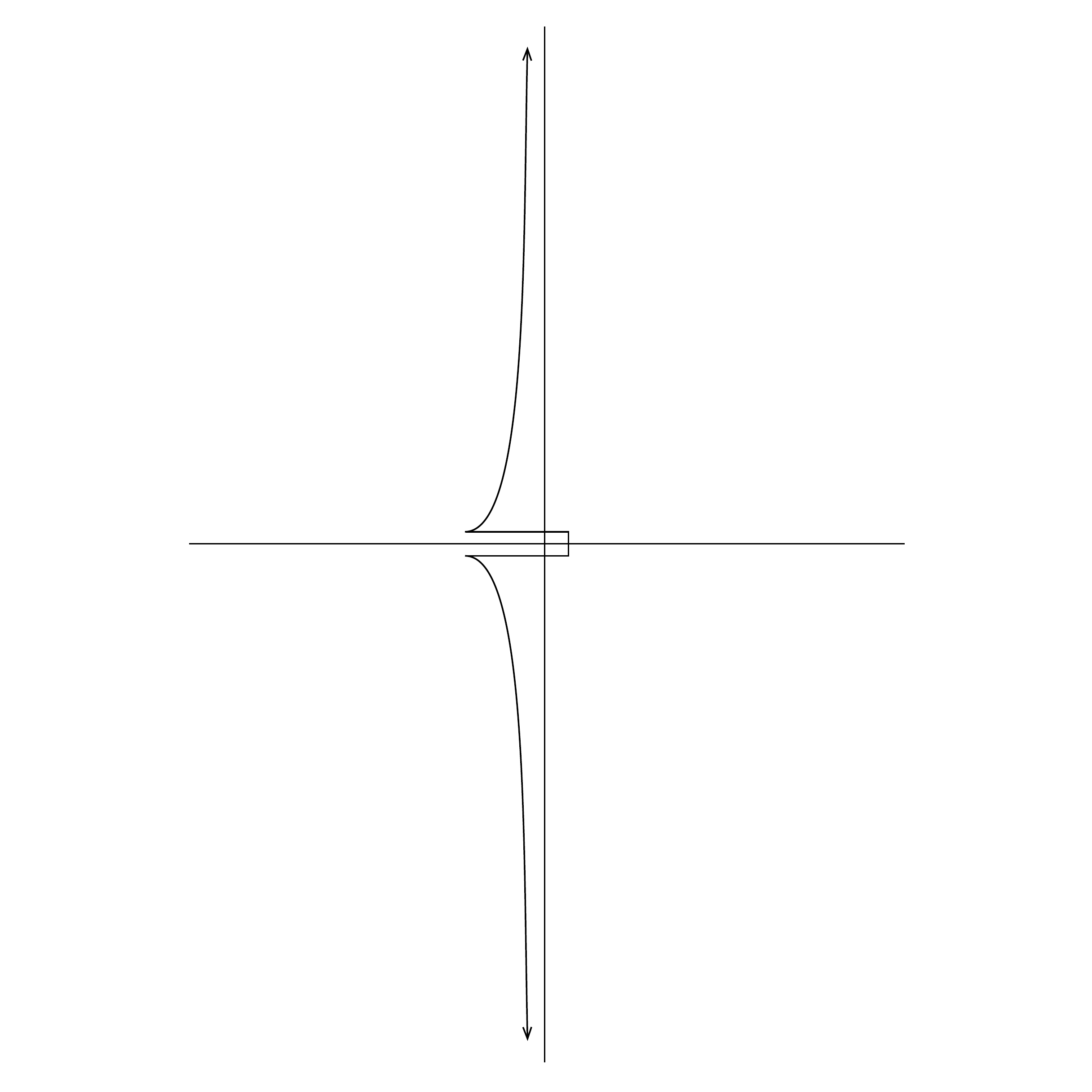}}
    
   \vskip-50pt $$\hskip -124pt\mathcal C_q$$
   \vskip 10pt
   \centerline{\hskip -100pt Figure 1}
   \vskip 20pt

  To construct the contour $\mathcal C_q$  first take a slit along the real axis from $1 - \frac{c_2}{\log q}$ to $1$ and construct a line just above and just below the slit. Then take two asymptotes to the line $\Re(s) = 1$ with the property that if $\sigma + it$ is on the asymptote and $|t| \ge 1$, then $\sigma$ satisfies (\ref{ZeroFreeRegionL}). If $q =1$, we do a similar construction using 
   (\ref{ZeroFreeRegionZ}).
  \vskip 20pt
  \section{Proof of theorem \ref{MainTheorem} }
  
  The proof of theorem \ref{MainTheorem} is based on the following theorem.

  \begin{theorem} \label{MainTheorem2} Let $h,q$ be fixed coprime integers with $q > 2$ and $\mu(q) \ne 0.$  Then for $x\to\infty$ there exist absolute constants $c,c' > 0$ such that
\begin{align*}
\sum_{n\le x} e^{2\pi i \frac{h A(n)}{q}} & = \frac{ \sin\left(\frac{\mu(q) \,\pi }{\phi(q)}   \right)}{\pi} \int\limits_{1-\frac{c }{\sqrt{\log x}}}^1 \left(\underset{\chi\ne\chi_0}{\prod_{\chi\hskip-5pt\pmod{q}}} \;L(\sigma, \,\overline{\chi})^{\frac{\tau(\chi)\overline{\chi(h)}}{\phi(q)}}\right) \cdot |\zeta(\sigma)|^{\frac{\mu(q)}{\phi(q)}}\cdot e^{H_{h,q}(\sigma)} \; \frac{x^\sigma}{\sigma} \; d\sigma\\
& \hskip 327pt+ \mathcal O\left( x e^{-c'\,\sqrt{\log x}}  \right).
\end{align*}
On the other hand if $\mu(q) = 0$, then
$ \sum\limits_{n\le x} e^{2\pi i \frac{h A(n)}{q}} =  \mathcal O\left( x e^{-c'\,\sqrt{\log x}}  \right).$

 \end{theorem}

    \begin{proof}  
  \vskip 10pt
  The proof of theorem \ref{MainTheorem2} relies on the following
 lemma  taken from \cite{Dav}. 
 \vskip 10pt
 \begin{lemma} \label{DavLemma} Let
 $$\delta(x) := \begin{cases} 0, & \text{if} \; 0 < x < 1\\
 \frac12, & \text{if} \; x = 1\\
 1, & \text{if} \; x > 1, \end{cases}$$
 then for $x, T > 0$, we have
 $$  \left | \frac{1}{2\pi i} \int\limits_{c-iT}^{c+iT}\;  \frac{x^s}{s}\; ds \; - \; \delta(x)\right | \; < \; \begin{cases}
 x^c\cdot \min\left(1, \frac{1}{T |\log x| } \right), & \text{if}\; x \ne 1,\\
 cT^{-1}, & \text{if} \; x = 1.
 \end{cases}
 $$
 \end{lemma}
 
 \vskip 10pt
  It  follows from lemma \ref{DavLemma}, for $x, T \gg 1$ and $c = 1 + \frac{1}{\log x}$, that
    \begin{equation} \label{countingpsi}
    \frac{1}{2\pi i}\int\limits_{c-iT}^{c+iT} L\big(s, \psi_{h/q}\big) \;\frac{x^s}{s} \; ds = \sum_{n\le x} \psi_{h/q}(n) \; + \; \mathcal O\left( \frac{x \log x}{T}  \right)
    \end{equation}
  
  Fix large constants $c_1, c_2 > 0.$ Next, shift the integral in (\ref{countingpsi}) to the left and deform the line of integration to a contour $$L^+ \; + \; \mathcal C_{T,x}\; + \; L^-$$ as in  figure 2 below which contains 
  two short horizontal lines: $$L^{\pm} = \bigg\{\sigma \pm iT \;\, \bigg | \;\,
 1 -\frac{c_1}{\log qT} \; \le \; \sigma \; \le \; 1+\frac{1}{\log x}\bigg\},$$ together
    together with the contour $C_{T,x}$ which is similar to $C_q$ except that the two curves asymptotic to the line $\Re(s) = 1$ go from $1 -\frac{c_1}{\sqrt{\log qT}} + iT$ to $1 -\frac{c_2}{\sqrt{\log x}} +i\varepsilon$ and   $1 -\frac{c_2}{\sqrt{\log x}} - i\varepsilon$ to $1 -\frac{c_1}{\sqrt{\log qT}} - iT$, respectively, for $0 <\varepsilon \to 0. $
 
 \vskip 25pt

   \centerline{\hskip-101pt\includegraphics[width=80mm]{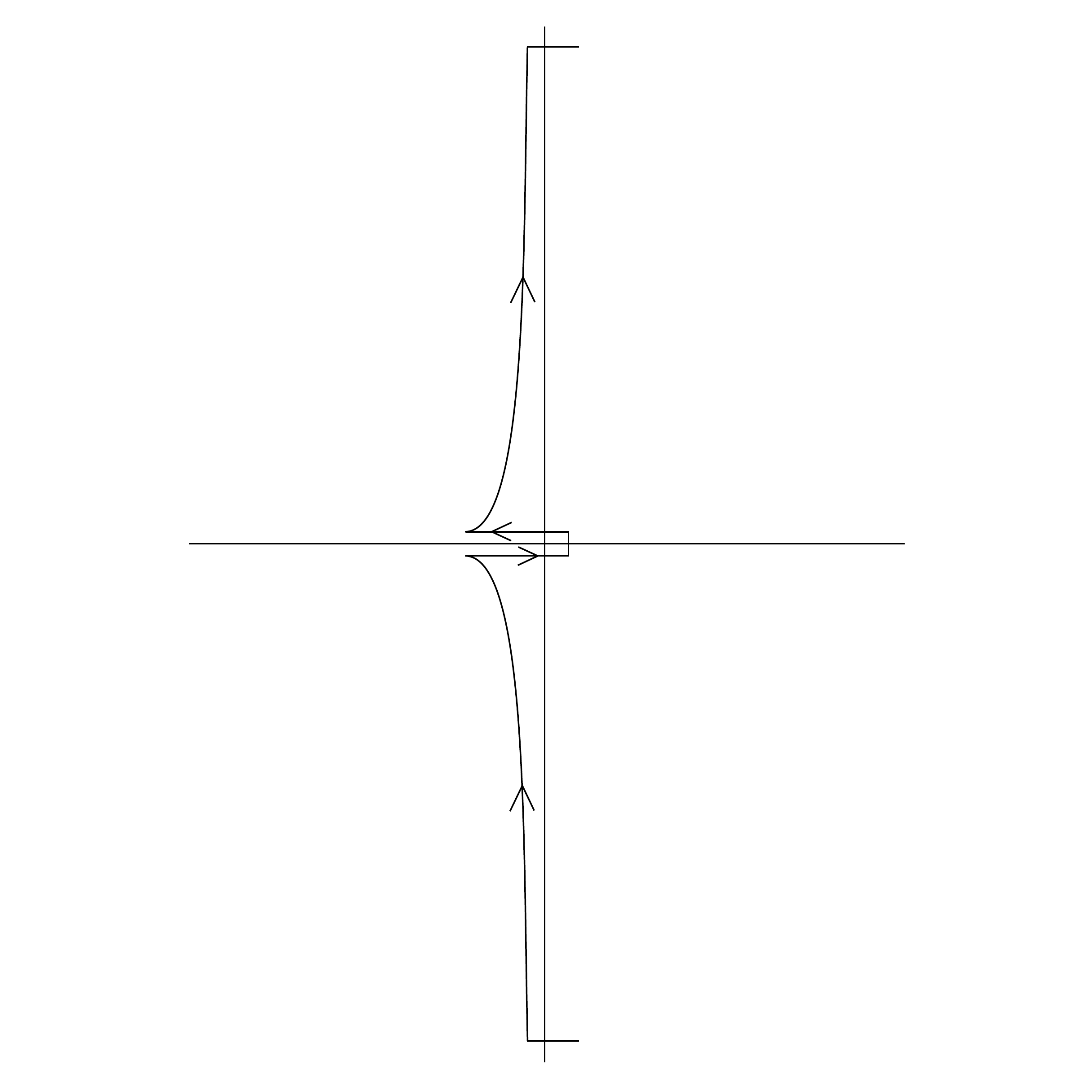}}
    \vskip -238.7pt $$\hskip -60pt L^+$$
   \vskip 132pt $$\hskip -142pt\mathcal C_{T,x}$$
  \vskip -2pt   $$\hskip -60pt L^-$$
   \vskip 12pt
   \centerline{\hskip -100pt Figure 2}
   \vskip 27pt

   Now, by the zero-free regions
 (\ref{ZeroFreeRegionL}),  (\ref{ZeroFreeRegionZ}), the region to the right of the contour 
$L^+ +  \mathcal C_{T,x} + L^-$ does not contain any branch points or poles of the L-functions $L(s, \chi)$ for any $\chi\hskip-3pt\pmod{q}$. 
    It follows that
    \begin{equation} \label{contourintegral}
    \frac{1}{2\pi i}  \int\limits_{c-iT}^{c+iT} L\big(s, \psi_{h/q}\big) \;\frac{x^s}{s} \; ds   \; = \;      \frac{1}{2\pi i}\left(\int_{L^+} + \int_{\mathcal C_\epsilon} + \int_{L^-}\right) L\big(s, \psi_{h/q}\big) \;\frac{x^s}{s} \; ds .
    \end{equation}
    
  The main contribution for the integral along  $L^+ +  \mathcal C_{T,x} + L^-$ in (\ref{contourintegral})  comes from the integrals along the straight lines above and below the slit on the real axis $\Big[1-\frac{c_2}{\sqrt{\log x}}, \;1\Big].$ These integrals cancel if the function $L\big(s, \psi_{h/q}\big)$ has no branch points or poles on the slit. It follows from theorem \ref{Identity} that this will be the case if $\mu(q) = 0$. The remaining integrals in \ref{contourintegral} can then be estimated as in the proof of the prime number theorem for arithmetic progressions (see \cite{Dav}), yielding an error term of the form  
  $ \mathcal O\left( x e^{-c'\,\sqrt{\log x}}  \right)$.  This proves the second part of theorem \ref{MainTheorem2}.
  
\vskip 10pt  
  
  Next, assume $\mu(q) \ne 0.$ In this case $L(s,\psi_{h/q})$ has a branch point at $s=1$ coming from the Riemann zeta function,  it is necessary to keep track of the change in argument. Let $0^+i$ denote the upper part of the slit and let $0^- i$ denote the lower part of the slit.  Then we have  $\log[\zeta(\sigma+0^+  i) =  \log |\zeta(\sigma)| - i\pi$ and
  $\log[\zeta(\sigma+0^- i) =  \log |\zeta(\sigma)| + i\pi$.
   
   By the standard proof of the prime number theorem for arithmetic progressions    
     it follows that (with an error  $\mathcal O\big(e^{-c'\sqrt{\log x}}   \big)$) the right hand side of (\ref{contourintegral}) is asymptotic to
\begin{equation} \label{SlitIntegral}
\mathcal I_{\text{slit}} :=\frac{-1}{2\pi i} \int\limits_{1-\frac{c}{\sqrt{\log x}}}^1 \left[\exp\Big(\log\left(L\left(\sigma + 0^+i, \; \psi_{h/q}  \right)\right)\Big) -  \exp\Big(\log\left(L\left(\sigma -0^-i, \; \psi_{h/q}  \right)\right)\Big)\right] \frac{x^\sigma}{\sigma} \; d\sigma. 
\end{equation} 
   
We may evaluate $I_{\text{slit}}$ using theorem \ref{Identity}. This gives
\begin{align*}
\mathcal I_{\text{slit}} & =  \;\frac{-1}{2\pi  i}  \int\limits_{1-\frac{c}{\sqrt{\log x}}}^1 \left(\underset{\chi\ne\chi_0}{\prod_{\chi\hskip-5pt\pmod{q}}} \;L(\sigma,\,\overline{\chi})^{\frac{\tau(\chi)\overline{\chi(h)}}{\phi(q)}}\right) \cdot e^{U_{h,q}(\sigma)}\\
&
\hskip 70pt \cdot \left[\exp\left(\frac{\mu(q)}{\phi(q)}\Big( \log|\zeta(\sigma)| -i\pi  \Big)\right)  \; - \;    \exp\left(\frac{\mu(q)}{\phi(q)}\Big( \log|\zeta(\sigma)| +i\pi  \Big)^{}    \right)\right] \frac{x^\sigma}{\sigma} \; d\sigma 
\\
& \\
& = \; \frac{\sin\left(\frac{\mu(q) \,\pi }{\phi(q)}   \right)}{\pi} \int\limits_{1-\frac{c}{\sqrt{\log x}}}^1 \left(\underset{\chi\ne\chi_0}{\prod_{\chi\hskip-5pt\pmod{q}}} \;L(\sigma, \,\overline{\chi})^{\frac{\tau(\chi)\overline{\chi(h)}}{\phi(q)}}\right) \cdot |\zeta(\sigma)|^{\frac{\mu(q)}{\phi(q)}}\cdot e^{U_{h,q}(\sigma)} \; \frac{x^\sigma}{\sigma} \; d\sigma.
\end{align*}
 \vskip 10pt  
   As in the previous case when $\mu(q) = 0,$ the remaining integrals in \ref{contourintegral} can then be estimated as in the proof of the prime number theorem for arithmetic progressions, yielding an error term of the form  
  $ \mathcal O\left( x e^{-c'\,\sqrt{\log x}}  \right)$. This completes the proof of theorem \ref{MainTheorem2}.
 \end{proof} 
 
 \vskip 10pt
 The proof of theorem \ref{MainTheorem} follows from theorem \ref{MainTheorem2} if we can obtain  an asymptotic formula for the integral
 \begin{equation} \label{SlitIntegral}
 \mathcal I_{\text{slit}}\; = \; \frac{\sin\left(\frac{\mu(q) \,\pi }{\phi(q)}   \right)}{\pi} \int\limits_{1-\frac{c}{\sqrt{\log x}}}^1 \left(\underset{\chi\ne\chi_0}{\prod_{\chi\hskip-5pt\pmod{q}}} \;L(\sigma, \,\overline{\chi})^{\frac{\tau(\chi)\overline{\chi(h)}}{\phi(q)}}\right) \cdot |\zeta(\sigma)|^{\frac{\mu(q)}{\phi(q)}}\cdot e^{U_{h,q}(\sigma)} \; \frac{x^\sigma}{\sigma} \; d\sigma.
 \end{equation}
 
 Since we have assumed $q$ is fixed, it immediately follows that for arbitrarily large $c \gg 1$ and $x \to \infty,$ we have
 $$I_{\text{slit}}\; = \; \frac{\sin\left(\frac{\mu(q) \,\pi }{\phi(q)}   \right)}{\pi}\hskip-10pt \int\limits_{1-\frac{c \log\log x}{\log x}}^1 \left(\underset{\chi\ne\chi_0}{\prod_{\chi\hskip-5pt\pmod{q}}} \;L(\sigma, \,\overline{\chi})^{\frac{\tau(\chi)\overline{\chi(h)}}{\phi(q)}}\right) \cdot |\zeta(\sigma)|^{\frac{\mu(q)}{\phi(q)}} \,\cdot \,e^{U_{h,q}(\sigma)} \; \frac{x^\sigma}{\sigma} \; d\sigma \;\, + \;\, \mathcal O\left( \frac{x}{(\log x)^c}  \right).$$
 
 \vskip 10pt
 Now, in the region $1-\frac{c \log\log x}{\log x} \le \sigma \le 1$,
 
 $$\underset{\chi\ne\chi_0}{\prod_{\chi\hskip-5pt\pmod{q}}} \;L(\sigma, \,\overline{\chi})^{\frac{\tau(\chi)\overline{\chi(h)}}{\phi(q)}}\cdot \frac{e^{H_{h,q}(\sigma)}}{\sigma} = \underset{\chi\ne\chi_0}{\prod_{\chi\hskip-5pt\pmod{q}}} \;L(1, \,\chi)^{\frac{\tau(\overline{\chi})\chi(h)}{\phi(q)}}\cdot e^{U_{h,q}(1)} \; + \; \mathcal O\left( \frac{\log\log x}{\log x}  \right).$$
 \vskip 8pt
 Consequently,
 \begin{align}
 I_{\text{slit}}\; & = \; \frac{ \sin\left(\frac{\mu(q) \,\pi }{\phi(q)}   \right)}{\pi} \underset{\chi\ne\chi_0}{\prod_{\chi\hskip-5pt\pmod{q}}} \;L(1, \,\chi)^{\frac{\tau(\overline{\chi})\chi(h)}{\phi(q)}}\cdot e^{U_{h,q}(1)} \hskip -8pt \int\limits_{1-\frac{c \log\log x}{\log x}}^1 \zeta(\sigma)^{\frac{\mu(q)}{\phi(q)}}\; x^\sigma \; d\sigma\nonumber\\
 & \hskip 140pt+ \mathcal O\left( \frac{\log\log x}{\log x}  \left|\;\;  \int\limits_{1-\frac{c \log\log x}{\log x}}^1 
 \zeta(\sigma)^{\frac{\mu(q)}{\phi(q)}}\; x^\sigma  \; d\sigma  \;\;\right|    \right).
 \label{IslitIntegral}\end{align}

 It remains to compute the integral of $|\zeta(\sigma)|^{\frac{\mu(q)}{\phi(q)}}$ occurring in  (\ref{IslitIntegral}). For $\sigma$ very close to 1, we have
 $$|\zeta(\sigma)|^{\frac{\mu(q)}{\phi(q)}} = \left(\frac{1}{|\sigma-1|} \;+\; \mathcal O(1)  \right)^{\frac{\mu(q)}{\phi(q)}} = \left( \frac{1}{|\sigma-1|}  \right)^{\frac{\mu(q)}{\phi(q)}} \; + \; \mathcal O\left( \left( \frac{1}{|\sigma-1|}  \right)^{\frac{\mu(q)}{\phi(q)}-1} \right).$$
 It follows that
 \begin{equation} \label{zetaintegral}
 \int\limits_{1-\frac{c \log\log x}{\log x}}^1 
 |\zeta(\sigma)|^{\frac{\mu(q)}{\phi(q)}}\; x^\sigma  \; d\sigma \; = \;  \Gamma\left(1 - \frac{\mu(q)}{\phi(q)}   \right) \; \frac{x}{(\log x)^{1-\frac{\mu(q)}{\phi(q)}}} \; + \; \mathcal O\left( \frac{x}{(\log x)^{2-\frac{\mu(q)}{\phi(q)}}}  \right).
 \end{equation}
 
 Combining equations  (\ref{IslitIntegral}) and (\ref{zetaintegral}) we obtain
 \begin{align*}
 I_{\text{slit}} & =\frac{\sin\left(\frac{\mu(q) \,\pi }{\phi(q)}   \right)}{\pi}\, \Gamma\left(1 - \frac{\mu(q)}{\phi(q)}   \right) \underset{\chi\ne\chi_0}{\prod_{\chi\hskip-5pt\pmod{q}}} L(1, \,\chi)^{\frac{\tau(\overline{\chi})\chi(h)}{\phi(q)}}\, e^{U_{h,q}
  (1)} \;  \frac{x}{(\log x)^{1-\frac{\mu(q)}{\phi(q)}}} \;+\; \mathcal O\left(  \frac{x}{(\log x)^{2-\frac{\mu(q)}{\phi(q)}}} \right).
 \end{align*}
 
 \vskip 10pt\noindent
 {\bf Remark:} As pointed out to me by G\' erald Tenenbaum, it is also possible to deduce  theorem 1.2 directly from theorem 2.3 by using theorem II.5.2 of [8].  In this manner one can obtain an explicit asymptotic expansion which, furthermore, is valid for values of $q$ tending to infinity with $x$. 
 
 \vskip 20pt
 \section{Examples of equidistribution (mod 3) and (mod 9)}
 \vskip 5pt\noindent
 {\bf Equidistribution (mod 3):}
   Theorem (\ref{MainTheorem})
     says that for $h = 1, \; q = 3:$
\begin{align*}
\sum_{n \, \le \, x} e^{\frac{2\pi i A(n)}{3}} \; & = \;{\frac{-V_{1,3}}{ \pi}\;\Gamma\left(\frac32\right)\underset{\chi\ne\chi_0}{\prod\limits_{\chi\hskip-5pt\pmod{3}}} \;L(1, \,\chi)^{ \frac{G(\overline{\chi})}{2}} } \frac{x}{(\log x)^{\frac32}}\;
{ \Big(1 + \mathcal O\left( \frac{1}{\log x}\right)\Big)}\\
& \approx \; (-0.503073 + 0.24042\, i) \, \frac{x}{(\log x)^{\frac32}}.\end{align*}

We computed the above sum for $x = 10^7$ and obtained
$$\sum_{n \, \le \, 10^7} e^{\frac{2\pi i A(n)}{3}} \approx -98,423.00 + 55,650.79\, i.$$
Our theorem predicts that
$$\sum_{n \, \le \, 10^7} e^{\frac{2\pi i A(n)}{3}} \; \approx \; -88,870.8 + 42,471.7 \;i.$$
 Since $\log\left(10^7\right) \approx 16.1$ is  small, this explains the discrepancy between the actual and predicted results.
\vskip 8pt
As $x \to \infty,$ we have
\begin{align*}
\underset{A(n)\; \equiv \; a \hskip-4pt\pmod{3}}{\sum_{n\, \le \, x} } \hskip-9pt 1 \; 
& = \; \frac13\sum_{h = 0}^2 \sum_{n\, \le \, x} e^{ \frac{2\pi i A(n)h}{3} } e^{-\frac{2\pi i h\,a}{3}}
\\
&  = \; \frac{x}{3} \; + c_a \; \frac{x}{(\log x)^{\frac32}} \; + \; \mathcal O\left( \frac{x}{(\log x)^{\frac52}}  \right)
\end{align*} 
where
$$c_0 = -0.335382, \qquad c_1 \; \approx \;0.306498, \qquad c_2 \; \approx \; 0.0288842.$$
 
 \vskip 8pt
 \noindent
  {\bf Equidistribution (mod 9):}
  \vskip 5pt
  Our theorem says that for $h \ne 3,6$ ($1\le h<9)$ and $q = 9$: 
 $$\sum_{n \, \le \, x} e^{\frac{2\pi i h A(n)}{9}}\; = \; \mathcal O\left(x\, e^{-c_0\sqrt{\log x}}   \right).$$

 Surprisingly!! there is a huge amount of cancellation when $x = 10^7:$

 $$  \sum_{n \, \le \, 10^7}  e^{\frac{2\pi i h A(n)}{9}}\; \approx \;
  \begin{cases} 
  -315.2 -140.4\, i & \text{if} \; h = 1,\\
  \;282.2 - 543.4\, i & \text{if} \; h = 2,\\
  \; 94.5 + 321.9 \, i  & \text{if} \; h = 4,\\
  \; 94.5 - 321.9 \, i  & \text{if} \; h = 5,\\
   \;282.2 + 543.4\, i & \text{if} \; h = 7,\\
    -315.2 +140.4\, i & \text{if} \; h = 8.\\
    \end{cases}$$

 \vskip 20pt
 \section*{Acknowledgement}
{\it The author would like to thank Ada Goldfeld for creating the figures in this paper and would also like to thank Wladyslaw Narkiewiz for pointing out the reference  \cite{Del}.}

\vskip 10pt
\end{document}